\documentclass[12pt,a4paper,reqno]{amsart}
\usepackage {amsfonts}

\newtheorem{theorem}{Theorem}
\newtheorem{lemma}[theorem]{Lemma}
\theoremstyle{remark}
\newtheorem{remark}[theorem]{Remark}

\DeclareMathOperator{\minor}{minor}
\DeclareMathOperator{\Trace}{Trace}

\author{I.G. Korepanov}
\thanks{South Ural State University, Chelyabinsk, Russia. E-mail: kig@susu.ac.ru}
\title[Geometric torsions and a topological field theory]{Geometric torsions and an Atiyah-style topological field theory}
\date{}

\newcommand{\kletka}{} 
\def\kletka#1{\vbox{\hsize 3cm \tolerance 10000 \footnotesize  \noindent #1}}

\newcommand{\sdvigvverh}{} 
\def\sdvigvverh#1{\begin{matrix}#1\\ \mbox{}\end{matrix}}

\begin{document}

\begin{abstract}
The construction of invariants of three-dimensional manifolds with a triangulated boundary, proposed earlier by the author for the case when the boundary consists of not more than one connected component, is generalized to any number of components. These invariants are based on the torsion of acyclic complexes of geometric origin. The relevant tool for studying our invariants turns out to be F.A.~Berezin's calculus of anti-commuting variables; in particular, they are used in the formulation of the main theorem of the paper, concerning the composition of invariants under a gluing of manifolds. We show that the theory obeys a natural modification of M.~Atiyah's axioms for anti-commuting variables.
\end{abstract}

\maketitle

\section{Introduction}

In this paper, we conclude the construction of one version of invariants of three-dimensional manifolds with triangulated boundary, started in paper~\cite{I}. To be more exact, we present a construction of these invariants for a compact connected three-dimensional manifold with an arbitrary number of boundary components, while in~\cite{I} we restricted ourselves to not more than one component.

Our invariants --- we call them ``geometric'' --- are built on the basis of torsions of acyclic complexes made of vector spaces which consist of differentials of geometric quantities assigned to the simplexes of a manifold triangulation, that is, vertices, edges and so on. We don't have a general recipe for such a geometrization, but in practice, it turns out that the point of at least very many homogeneous spaces of Lie groups can be used as ``coordinates'' of triangulation vertices, and then build a complex using intuitive, although typically not very complicated, considerations. In the present paper, we are using the group~$\mathrm E(3)$ of motions of the three-dimensional Euclidean space~$\mathbb R^3$, and this space itself as homogeneous space; of other geometries which can be used for three-dimensional triangulated manifolds, we can mention, if restrict ourselves to the already written papers, the \emph{four-dimensional} Euclidean geometry~\cite{3man4geo} and \emph{two-dimensional} affine volume-preserving geometry~\cite{SL2,SL2'}. There are also much more available unpublished examples, in particular, the paper~\cite{kkm} is in preparation, where we are using the group $\mathrm{PSL}(2,\mathbb C)$ of fractional-linear transformations of the complex plane; there are also ``noncommutative'' versions of some theories, where elements of an associative algebra are used for ``coordinates'' of vertices, and nothing more is required than the invertibility of some combinations of those elements. For four-dimensional manifolds, we have proposed a construction of a complex based on Euclidean geometry which was as well four-dimensional~\cite{33,24,15}, and there is also the paper~\cite{kiev} where a formula is proved indicating the possibility of using three-dimensional affine volume-preserving geometry.

Recall that \emph{Pachner moves} are elementary rebuildings of a triangulation of a piecewise-linear manifold; using a long enough chain of such moves, one can pass from one triangulation to any other. The remarkable fact observed in every specific case but not yet having a general explanation is that the torsion of an acyclic complex of geometric origin is transformed in a simple multiplicative way under Pachner moves. The typical case is when the torsion is multiplied by some powers of volumes of simplexes arising under the Pachner move, and divided by the same powers of volumes of disappearing simplexes. Thus, torsion divided by the same powers of volumes of \emph{all} participating simplexes is an invariant of all Pachner moves, the examples of which in this paper are formulas (\ref{inv-closed}) and~(\ref{inv-b}). We recall also that the starting point of our theory~\cite{3dcase} consisted exactly in specific formulas related with Pachner moves, and in the first place, with move $2\to 3$ (two tetrahedra with a common face are replaced with three tetrahedra with a common edge) for three-dimensional triangulated manifolds.

In paper~\cite{I}, we have demonstrated how to construct not one but many algebraic complexes for a manifold with a one-component boundary, on which (boundary) a fixed triangulation is given. Such complex depends on two arbitrary sets of boundary edges $\mathcal C$ and~$\mathcal D$ of equal cardinality; we have shown on examples that such complex is acyclic in many enough cases to arouse interest (and if it is not acyclic, its torsion is assumed to be zero). Under a change of boundary triangulation, the invariants undergo a linear transform. Moreover, it proved possible to write out a formula expressing the invariant of the closed manifold obtained by gluing two manifolds with boundary (The boundaries are identical and identically triangulated, but oppositely oriented) in terms of the invariants of these latter. The formula had the form of a scalar product, which corresponds to the known axioms of M.~Atiyah~\cite{atiyah,atiyah1} for a topological quantum field theory. By the time of writing the paper~\cite{I}, we could not handle the case of a multicomponent boundary.

In this paper we show how to overcome this difficulty (by considering ``sways'' of each boundary component as a whole, see section~\ref{sec:many-components}). In doing this, it becomes clear that it is convenient to unite all invariants in a generating function involving anti-commuting (Grassmann) variables: this corresponds adequately, in particular, to the sign change in the invariant under a change of edge ordering. The real role of anti-commuting variables comes out, however, when we write the formula for the composition of invariants under a gluing of manifolds by some boundary components: the result is, up to a factor, the Berezin integral of the product of generating functions for the two manifolds in the variables corresponding to the glued boundary components! Thus, the obtained theory does not obey Atiyah's axioms literally, but does obey quite well their natural modification for the case of Grassmann variables. We describe this modification of axioms; an interesting distinction of the Grassmann case from the usual one turns out to be the relation between the invariants of spaces $\Sigma\times I$, $\Sigma\times S^1$ and the dimensionality of the linear space corresponding to~$\Sigma$, where $\Sigma$ is a two-dimensional surface, $I=[0,1]$ is a segment, and $S^1$ is a circle. We describe this relation briefly, leaving its detailed study for a separate work.

The contents of the rest of the sections of this paper is as follows. In section~\ref{sec:many-components} we present our construction for a manifold with any number of boundary components, recalling in passing the necessary basic notions. In section~\ref{sec:Berezin-calculus} we introduce the generating function of Grassmann variables for the obtained invariants and show that, at this stage already, the Berezin integral is present. In section~\ref{sec:gluing} we prove the theorem about the composition of invariants under a gluing of manifolds. In section~\ref{sec:Atiyah-axioms} consider the modification of Atiyah's axioms. In section~\ref{sec:discussion} we discuss the results and plans for further research.

\section{The complex and invariants for a manifold with any number of boundary components}
\label{sec:many-components}

In this section, we present a construction of an algebraic complex and invariants for a connected triangulated compact oriented three-dimensional manifold~$M$ whose boundary has $m$ connected components, $m=0,1,\ldots$. While we present new, with regard to paper~\cite{I}, material in a rather detailed way, we just briefly remind the facts already written in paper~\cite{I} and propose the reader to consult it when necessary.

We begin with \emph{geometrization} of triangulation simplexes, to be exact, with assigning to its every vertex~$A$ three real numbers $x_A,y_A,z_A$ which we are going to call its unperturbed Euclidean coordinates. These numbers are arbitrary, with the only requirement that they should lie in a general position with regard to all further constructions. In this way, Euclidean geometry is introduced in every tetrahedron, although this by no means applies to the whole~$M$, because the resulting tetrahedra in Euclidean space~$\mathbb R^3$ are allowed to intersect.

We use (in particular) the following vector spaces in the construction of our complexes:
\begin{itemize}
\item Lie algebra $\mathfrak e(3)$ of infinitesimal isometries of~$\mathbb R^3$,
\item space $(dx)$ of column vectors of coordinate differentials of triangulation vertices (which vertices exactly, we will specify in every individual case),
\item space $(dl)$ of column vectors of edge length differentials,
\item space $(d\omega, d\alpha)$ of column vectors of deficit angles at inner edges differentials~$d\omega$ and differentials~$d\alpha$ of minus inner dihedral angles at boundary edges (which edges exactly are used in this and previous items, we will also specify in every individual case).
\end{itemize}

We will need linear mappings $f_1$, $f_2$ and~$f_3$ between these vector spaces, these mappings being the differentials of ``macroscopic'' mappings having a simple geometric meaning: $f_1=dF_1$, $f_2=dF_2$, $f_3=dF_3$, where $F_1$ shows where the vertex coordinates are moved from their unperturbed locations under the action of an element of group of isometries~$\mathrm E(3)$ (so, $F_1$ acts from $\mathrm E(3)$ to the set of coordinates, while unperturbed coordinates of all vertices are parameters of the theory); $F_2$~builds edge lengths from given vertex coordinates (``perturbed''); $F_3$~builds deficit and dihedral 
from given edge lengths.

In the algebra $\mathfrak e(3)$, considered as linear space, one can choose a natural basis of three infinitesimal translations and three rotations, which allows to consider this algebra, like our other spaces, as consisting of column vectors, and identify mappings $f_1$, $f_2$ and~$f_3$ with matrices. If we also take into account that the partial derivatives, of which matrix~$f_3$ is built, possess symmetry properties~\cite{3dcase}:
$$
\frac{\partial\omega_i}{\partial l_j} = \frac{\partial\omega_j}{\partial l_i},\quad
\frac{\partial\alpha_i}{\partial l_j} = \frac{\partial\alpha_j}{\partial l_i},
$$
then we can prolong the chain of mappings $f_1,f_2,f_3$ and obtain, for example, an algebraic complex for a manifold \emph{without boundary} (thus no angles~$d\alpha$), written out, for instance, in formula~(3) 
of paper~\cite{I}:
\begin{equation}
0\longrightarrow \mathfrak e(3) \stackrel{f_1}{\longrightarrow} (dx) \stackrel{f_2}{\longrightarrow} (dl) \stackrel{f_3=f_3^{\rm T}}{\longrightarrow} (d\omega) \stackrel{f_4=-f_2^{\rm T}}{\longrightarrow} (dx^*) \stackrel{f_5=f_1^{\rm T}}{\longrightarrow} \mathfrak e(3)^* \longrightarrow 0.
\label{complex-closed}
\end{equation}
Defining the \emph{torsion} of complex~(\ref{complex-closed}) by formula
\begin{equation}
\tau \stackrel{\rm def}{=} \frac{\minor f_1 \, \minor f_3 \, \minor f_5}{\minor f_2 \, \minor f_4},
\label{tau}
\end{equation}
where the choice of minors is explained in subsection~2.2 
of paper~\cite{I},\footnote{We will soon consider in detail the questions of choosing the minors, and in a more general case.} we get an invariant of a manifold without boundary from formula
\begin{equation}
I(M)=\frac{\tau \prod_{\textrm{over all tetrahedra}}(-6V)}{\prod_{\textrm{over all edges}}l^2}.
\label{inv-closed}
\end{equation}
Here $V$ is the oriented tetrahedron volume, and $l$ --- edge length.

Now we show how to generalize this construction for a manifold with an arbitrary number $m$ of boundary components. The main justification of such exactly generalization will be the fact that, as we will see in section~\ref{sec:gluing}, an analogue of theorem~15 
from paper~\cite{I} holds stating that the invariants of the result of the gluing of manifolds $M_1$ and~$M_2$ by some components of their boundaries can be obtained from the invariants of $M_1$ and~$M_2$. Our construction will embrace in a uniform way the cases $m=0$ and $m=1$ as well, to which separate sections in~\cite{I} were devoted. Here is how our algebraic complex looks like:
\begin{multline}
0\longrightarrow \mathfrak e(3) \stackrel{f_1}{\longrightarrow} \begin{pmatrix} dx_{\textrm{inner}}\\ \oplus\\ m\mathfrak e(3) \end{pmatrix} \stackrel{f_2}{\longrightarrow} (dl)\\ \stackrel{f_3}{\longrightarrow} (d\omega, d\alpha) \stackrel{f_4}{\longrightarrow} \begin{pmatrix} dx_{\textrm{inner}}^*\\ \oplus\\ m\mathfrak e(3)^* \end{pmatrix} \stackrel{f_5}{\longrightarrow} \mathfrak e(3)^* \longrightarrow 0.
\label{complex-b}
\end{multline}

Now we describe in detail the meaning of all symbols in~(\ref{complex-b}). Algebra $\mathfrak e(3)$, coming after ``$0\to $'', is the algebra of infinitesimal motions of Euclidean~$\mathbb R^3$, as before. But an important innovation occurs in the next term $\left( \begin{smallmatrix} dx_{\textrm{inner}}\\ \oplus\\ m\mathfrak e(3) \end{smallmatrix} \right)$. Here ``$dx_{\textrm{inner}}$'' means the collection of all  \emph{inner} vertex coordinate 
differentials, while only Euclidean motions of every boundary component \emph{as a whole} are permitted, by definition, to vertices lying on the boundary. We call such motions \emph{sways} of boundary components. To an infinitesimal sway of every component, an element of the very same algebra~$\mathfrak e(3)$ clearly corresponds, hence our notation ``$m\mathfrak e(3)$'' for the set of all sways. By definition, $f_1$ sends an element $a\in\mathfrak e(3)$ into $m$ identical sways of boundary components given by the same element~$a$, and into such values $dx_{\textrm{inner}}$ which are vertex displacements under the infinitesimal motion~$a$.

We choose now, like in paper~\cite{I}, two sets of boundary edges $\mathcal C$ and~$\mathcal D$ of the same cardinality (summed over the boundary components, the numbers of edges in $\mathcal C$ and~$\mathcal D$ can differ for each separate component). The next vector space~$(dl)$ in complex~(\ref{complex-b}) consists, by definition, of column vectors of length differentials for all inner edges, as well as boundary edges from set~$\mathcal C$. Mapping~$f_2$ is defined geometrically: independent displacements~$dx$ of inner vertices, together with boundary component sways, clearly determine changes of lengths~$dl$ (and, certainly, we get $dl_i=0$ if edge~$i$ lies on the boundary).

Similarly, vector space~$(d\omega, d\alpha)$ consists by definition of column vectors made of infinitesimal deficit angles~$d\omega_i$ for all inner edges~$j$ and of differentials of minus dihedral angles~$d\alpha_k$ for $k\in \mathcal D$.

Then, the further continuation of complex~(\ref{complex-b}) goes by symmetry, like it was for complex~(\ref{complex-closed}); we have only to mention that not simply~$-f_2^{\rm T}$ is taken for~$f_4$, but also the edge set $\mathcal C$ must be replaced by~$\mathcal D$.

The next theorem justifies the use of words ``algebraic complex'' to the sequence of spaces and mappings~(\ref{complex-b}).

\begin{theorem}
\label{th:complex-b}
The following identities hold for sequence~(\ref{complex-b}):
$$
f_2\circ f_1=0, \quad f_3\circ f_2=0, \quad f_4\circ f_3=0 \quad \textrm{and} \quad f_5\circ f_4=0.
$$
\end{theorem}

\emph{Proof} goes similarly to the proof of theorem~6 
of paper~\cite{I}, with just one new point: in~\cite{I}, the composition $f_3\circ f_2$ gave zero~$d\alpha$ because the boundary was motionless, while now we get the same in the situation where each boundary component can move, but only as a whole.
\qed

\smallskip

The torsion~$\tau$ of complex~(\ref{complex-b}) is defined by the old formula~(\ref{tau}). It is convenient to assume that we choose the minors entering it (that is, a~$\tau$-chain in complex~(\ref{complex-b})) in some standard way, and introduce in this connection the notion of $\emph{minimal rigid construction}$ of triangulation edges of a three-dimensional manifold with boundary~$M$. We will also introduce a similar notion for a connected surface --- boundary component of~$M$ --- for our further needs.

We choose three \emph{inner triangulation vertices} $A$, $B$ and~$C$.\footnote{In practical computations, for instance, in papers~\cite{dkm,I}, triangulations not having inner vertices can well be used. Our invariants, however, do not depend on a triangulation of the manifold interior, which allows us to assume that the three inner vertices do exist.} We take for the minor of matrix~$f_1$ its rows corresponding to coordinate differentials $dx_A,dy_A,\allowbreak dz_A,\allowbreak dy_B,\allowbreak dz_B,dz_C$. We deal in a symmetrical way with~$f_5$, that is, take for the minor of this matrix \emph{columns} corresponding to the same differentials.

Then we must choose edges corresponding to the rows of minor of~$f_2$; the same edges will correspond to the columns of minor of~$f_4$. This will be some minimal set of inner 
edges, such that if we fix their lengths and impose the additional condition of boundary component rigidity, \emph{all} vertices of~$M$ can move in the Euclidean~$\mathbb R^3$ only as a whole (the condition of boundary component rigidity can also be understood as fixing the lengths of all its edges). We say that such edges form a minimal rigid construction, and assume it to be fixed for a given triangulated~$M$.

Acting this way, we get all minors, except, maybe, $\minor f_3$, nonzero. We have to use for this latter the remaining inner edges, and sets $\mathcal C$ and~$\mathcal D$, according to what was explained above; if the complex comes out non-acyclic, $\minor f_3=0$.

In section~\ref{sec:gluing} we will need also a similar minimal rigid construction for a triangulated surface~$\Gamma$ --- the boundary component of manifolds $M_1$ and~$M_2$ by which they are glued into a new manifold~$M$. By definition, it consists of the minimal number of edges of~$\Gamma$ such that if their lengths are fixed, then all vertices in~$\Gamma$ can move in the Euclidean~$\mathbb R^3$ only as a whole.

\begin{remark}
Of course, we have already worked with minimal rigid constructions in paper~\cite{I}, not using, however, this term for them. It was exactly such a construction for a three-dimensional manifold that was used when describing minors of $f_2$ and~$f_4$ between theorems 6 
and~7, 
and for a one-component boundary --- in section~5 
of paper~\cite{I}.
\end{remark}

If edge sets $\mathcal C$ and~$\mathcal D$ coincide, no problems arise with the sign of~$\tau$, given the symmetry of the complex (compare remark~5 
in~\cite{I}). If, however, $\mathcal C$ and~$\mathcal D$ are different, then the sign does depend on the order in which we take edges in $\mathcal C$ and~$\mathcal D$, thus, these sets must be regarded as \emph{ordered}. The examination of this question involves far-going consequences, with which we will deal starting from section~\ref{sec:Berezin-calculus}.

We define the invariant corresponding to edge sets $\mathcal C$ and~$\mathcal D$ by the old formula~(10) 
in~\cite{I}:
\begin{equation}
I_{\mathcal C,\mathcal D}(M)=\frac{\tau \prod_{\textrm{over all tetrahedra}}(-6V)}{\prod_{\textrm{over inner edges}}l^2}
\label{inv-b}
\end{equation}
It is not hard to see that formula~(\ref{inv-b}) yields, in the cases $m=0$ and $m=1$, our old invariants from paper~\cite{I}.

\section{Geometric invariants and the Berezin calculus of anti-commuting variables}
\label{sec:Berezin-calculus}

Recall~\cite{B} that \emph{Grassmann algebra} over field~$\mathbb R$ or~$\mathbb C$ is an associative algebra with unity, having generators~$a_i$ and relations
$$
a_i a_j = -a_j a_i, \quad \textrm{including} \quad a_i^2 =0.
$$
Any element of a Grassmann algebra is a polynomial of degree $\le 1$ in each~$a_i$; in particular, such are the exponents encountered below, defined by the usual Taylor series, of which only a finite number of terms remains. The dimensionality of Grassmann algebra as a linear space is~$2^N$, where~$N$ is the number of generators, which we always assume to be finite.

The \emph{Berezin integral}~\cite{B} in a Grassmann algebra is defined by equalities
\begin{equation}
\int da_i =0, \quad \int a_i\, da_i =1, \quad \int gh\, da_i = g \int h\, da_i,
\label{integral_Berezina}
\end{equation}
if $g$ does not depend on~$a_i$ (that is, generator $a_i$ does not enter the expression for~$g$); multiple integral is understood as iterated one.

We turn now to formula~(\ref{inv-b}) for the invariant~$I_{\mathcal C,\mathcal D}(M)$. In it, $\minor f_3$ contains columns and rows corresponding to inner and boundary edges. The set~$\mathcal E_{\textrm{inner}}$ of such inner edges can be regarded as fixed for the given triangulation (not depending on $\mathcal C$ and~$\mathcal D$ and identical for rows and columns) --- these are all inner edges minus those involved in minors of $f_2$ and~$f_4$ according to section~\ref{sec:many-components}, that is, $\mathcal E_{\textrm{inner}}$ consists of inner edges not entering the minimal rigid construction described there. As concerns the boundary~$\partial M$, we, having in mind further applications,\footnote{namely, gluing of manifolds by boundary components in section~\ref{sec:gluing}} consider a minimal rigid construction in its every component as well. We denote the set of the boundary component edges \emph{not entering it} as~$\mathcal E_{\textrm{boundary}}^{(k)}$, where $k$ is the number of the component, and in the sequel we take as sets $\mathcal C$ and~$\mathcal D$ only subsets of the union of all sets~$\mathcal E_{\textrm{boundary}}^{(k)}$.

We assign to each edge~$i$ from the union of sets
$$
\mathcal E = \mathcal E_{\textrm{inner}} \cup \bigcup_{k=1}^m \mathcal E_{\textrm{boundary}}^{(k)}
$$
\emph{two} Grassmann generators $a_i$ and~$a_i^*$. Denote the submatrix of $f_3$, whose rows and columns correspond to edges in~$\mathcal E$, as~$\tilde f_3$. The generators without stars will correspond to the rows of~$\tilde f_3$ (and, accordingly, to differentials~$dl_i$), while those with stars --- to its columns (and, accordingly, to~$d\omega_i$ or~$d\alpha_i$). We compose the bilinear form
\begin{equation}
f(a,a^*)=\sum_{i,j} (\tilde f_3)_{ij}a_ia_j^*,
\label{f}
\end{equation}
where $a$ denotes the collection of all~$a_i$, while $a^*$ --- of all~$a_j^*$.

It is not hard to verify that the minor of~$\tilde f_3$ consisting of rows with numbers $i_1,\ldots,i_n$ and columns $j_1,\ldots,j_n$ (taken in this exactly order) is equal to the coefficient at $\prod_{k=1}^n a_{i_k} a_{j_k}^*$ in the polynomial $\exp f(a,a^*)$. The coefficients in this polynomial for different numbers of $i$'s and $j$'s are zero. 
Thus, it is natural to call it \emph{generating function} for minors of~$\tilde f_3$.

In general, we call generating function in this paper a polynomial of two sets of Grassmann generators $a_1,\ldots,a_N$ and~$a_1^*,\ldots,a_N^*$, whose coefficient at $\prod_{k=1}^n a_{i_k}a_{j_k}^*$ is equal to the quantity in question, attributed to \emph{ordered} sets $\mathcal C=(i_1,\ldots,i_n)$ and $\mathcal D=(j_1,\ldots,j_n)$, while coefficients at the monomials with different numbers of $a$ and~$a^*$ vanish. If $\mathcal C$ and~$\mathcal D$ are sets of \emph{only boundary} edges, as in section~\ref{sec:many-components}, and it is required to compose the generating function $\Phi (a_{\textrm{boundary}},a_{\textrm{boundary}}^*)$ of minors of~$\tilde f_3$ containing \emph{all} rows and columns from $\mathcal E_{\textrm{inner}}$, as well as some columns from~$\mathcal C$ and rows from~$\mathcal D$, then it is possible to do this by ``eliminating the dependence on variables at inner edges'' using Berezin integral:
\begin{equation}
\Phi (a_{\textrm{boundary}},a_{\textrm{boundary}}^*) = \int \exp f(a,a^*) \prod_{i\in \mathcal E_{\textrm{inner}}} da_i^*\, da_i.
\label{Phi}
\end{equation}
The validity of equality (\ref{Phi}) is easily established using~(\ref{integral_Berezina}).

Summing uo this section, we write out in the following theorem the generating function for invariants~(\ref{inv-b}).

\begin{theorem}
The generating function for invariants $I_{\mathcal C,\mathcal D}(M)$ is
\begin{multline}
\mathbf I_M(a_{\textrm{boundary}},a_{\textrm{boundary}}^*) \\
= \frac{\minor f_1\,\minor f_5}{\minor f_2\,\minor f_4}\, \frac{\prod_{\textrm{over all tetrahedra}}(-6V)}{\prod_{\textrm{over inner edges}}l^2}\, \Phi (a_{\textrm{boundary}},a_{\textrm{boundary}}^*),
\label{I_M}
\end{multline}
where $\Phi (a_{\textrm{boundary}},a_{\textrm{boundary}}^*)$ is defined by formulas (\ref{Phi}) and~(\ref{f}).
\qed
\end{theorem}

\section{Theorem on the composition of invariants under a gluing of manifolds}
\label{sec:gluing}

The general case of gluing several manifolds by some of their boundary components can be reduced to a chain of the following two operations:
\begin{itemize}
\item[(a)] gluing of two connected manifolds by one boundary component and
\item[(b)] gluing of two identical but oppositely oriented boundary components of one connected manifold.
\end{itemize}
Recall that we assume the manifolds to be compact and oriented.

Consider operation~(a). Let manifolds $M_1$ and~$M_2$ be glued into manifold~$M$ by a common boundary component~$\Gamma$. Consider what this implies for matrices of which complexes~(\ref{complex-b}) are composed for $M_1$ and~$M_2$. To begin, we note that matrices of partial derivatives $\bigl(\partial(\omega\textrm{ or }\alpha)_i / \partial l_j\bigr)$ add up, each of the matrices for 
$M_1$ and~$M_2$ is completed with zeros in rows and columns corresponding to edges not contained in the respective manifold. This is due to the fact that minus dihedral angles $(\alpha_i)_{M_1}$ and~$(\alpha_i)_{M_2}$ belonging to the same edge~$i$ lying in~$\Gamma$ but to different manifolds $M_1$ and~$M_2$ add up into a deficit angle~$\omega_i$ (while to every deficit angle the abovementioned zero is added). This will bring about soon, in its turn, a Berezin integral in variables corresponding to edges in~$\mathcal E_{\Gamma}$, where $\mathcal E_{\Gamma}$ is the set of all edges in~$\Gamma$ minus minimal rigid construction.

The rows of the submatrix of~$f_2$ corresponding to the minor in a complex like~(\ref{complex-b}) always belong to a minimal rigid construction of inner edges, joining all inner vertices and also three ones in every boundary component, as explained in section~\ref{sec:many-components}. The columns of the mentioned submatrix of~$f_2$ correspond to all inner vertex coordinates except six of them, and to all boundary component sways, according to section~\ref{sec:many-components}. Minimal rigid construction for~$M$ is obtained as the union of minimal rigid constructions for $M_1$ and~$M_2$ and also for~$\Gamma$. Thus, $\minor f_2$ for~$M$ contains the same rows as the minors for $M_1$ and~$M_2$, plus rows corresponding to edges in~$\Gamma$. As for the columns of the minor for~$M$, we also divide them into three groups. The first corresponds to all coordinate differentials for vertices in~$\Gamma$, except six of them: surface~$\Gamma$ lies in the interior of~$M$, so we have the right to choose six coordinates, whose columns do not enter in~$\minor f_2$, to belong to vertices in~$\Gamma$. Let these be $dx_A, dy_A,\allowbreak dz_A,\allowbreak dy_B,\allowbreak dz_B, dz_C$. The second group corresponds to all coordinate differentials for inner vertices of~$M_1$ and sways of all boundary components of~$M_1$, except~$\Gamma$. Similar columns belonging to~$M_2$ form the third group.

Thus, the following block structure arises for $\minor f_2$ entering in formula~(\ref{I_M}):
$$
\begin{array}{c|c|c|c}
&\kletka{All vertex coordinates in $\Gamma$, except $dx_A,\allowbreak dy_A,\allowbreak dz_A,\allowbreak dy_B,\allowbreak dz_B,\allowbreak dz_C$}&\kletka
{Inner vertex coordinates in $M_1$ and sways of boundary components of $M_1$, except $\Gamma$}&\kletka{Inner vertex coordinates in $M_2$ and sways of boundary components of $M_2$, except $\Gamma$}
\\ \hline 
\kletka{Minimal rigid construction in $\Gamma$}&\sdvigvverh *&\sdvigvverh 0&\sdvigvverh 0
\\ \hline 
\kletka{Minimal rigid construction in $M_1$}&\sdvigvverh *&\sdvigvverh *&\sdvigvverh 0
\\ \hline 
\kletka{Minimal rigid construction in $M_2$}&\sdvigvverh *&\sdvigvverh 0&\sdvigvverh *
\end{array}
$$

The diagonal blocks standing here at the second and third places can be chosen for minors of~$f_2$ in complexes~(\ref{complex-b}) written for $M_1$ and~$M_2$. This choice of minors is somewhat different from the ``standard'' one described after theorem~\ref{th:complex-b}; the corresponding submatrices of matrices~$f_1$, of which the minors must be taken, are now identities: they send $a\in \mathfrak e(3)$ into the same element~$a$ understood as a sway of~$\Gamma$. Certainly, we choose for $f_4$ and~$f_5$ the symmetric minors. After this, it is not hard to verify the formula expressing the generating function for invariants~$M$ in terms of functions for $M_1$ and~$M_2$, which deserves to be formulated as the following theorem.

\begin{theorem}
\begin{equation}
\mathbf I_M = \frac{(-1)^{N_{\Gamma}} \tau_{\Gamma}^2}{\prod_{\textrm{over edges in }\Gamma} l^2} \int \mathbf I_{M_1} \mathbf I_{M_2} \prod_{i\in \mathcal E_{\Gamma}} da_i^*\, da_i.
\label{skleika}
\end{equation}
Here $\tau_{\Gamma}=\frac{\minor g_1}{\minor g_2}$ is the torsion of the acyclic complex
$$
0 \longrightarrow \mathfrak e(3) \stackrel{g_1}{\longrightarrow} (dx)_{\Gamma} \stackrel{g_2}{\longrightarrow} (dl)_{\textrm{over m.r.c. in }\Gamma} \longrightarrow 0,
$$
$N_{\Gamma}$ is the number of vertices in $\Gamma$, ``m.r.c. in~$\Gamma$'' is the minimal rigid construction in~$\Gamma$ that we are using, and $\mathcal E_{\Gamma}$ is the set of edges in~$\Gamma$ not entering in it.
\qed
\end{theorem}


Now we consider operation~(b) --- the gluing of two boundary components of the same manifold --- one more elementary gluing operation indicated in the beginning of this section.

\begin{lemma}
\label{lemma:0}
The generating function~(\ref{I_M}) for invariants of manifold~$M$ obtained as a result of operation~(b) is identical zero.
\label{lemma:petlya-0}
\end{lemma}

\begin{proof}
The lemma follows from the fact that complex~(\ref{complex-b}) cannot be exact in the term~$(dl)$, because there exist nontrivial variations of edge lengths that leave deficit angles~$\omega$ zero and minus dihedral angles at boundary edges~$\alpha$ unchanged, 
and which cannot be obtained from any vertex coordinate changes and/or sways of boundary components of~$M$. These edge length variations are constructed as follows. For any (closed) loop in~$M$, its integer-valued intersection index with~$\Gamma$ is defined, where surface $\Gamma$ is now the result of gluing two boundary components taken with an arbitrary fixed orientation. It is not hard to produce a loop for which this intersection index equals~$1$. Thus, the intersection index defines an epimorphism $\pi_1(M)\to \mathbb Z$. Consider the corresponding covering~$\tilde M$ of manifold~$M$. 
In it, group~$\mathbb Z$ naturally acts. We place the vertices of~$\tilde M$ into the Euclidean~$\mathbb R^3$ in such way that the location of any vertex~$A'$ in the triangulation of~$\tilde M$, obtained from another vertex~$A$ by the action of element $1\in \mathbb Z$, should be obtained from the location of~$A$ by a Euclidean motion $g\in \mathrm E(3)$ common for all vertices, and any component of~$\partial \tilde M$ should undergo only a Euclidean motion \emph{as a whole} with respect to its ``unperturbed'' location (caused by the initial location of vertices in~$M$ in~$\mathbb R^3$). Changing~$g$, we get the desired nontrivial changes of edge lengths.
\end{proof}

In accordance with this, the natural analogue of formula~(\ref{skleika}) gives also a zero result, as the following lemma shows.

\begin{lemma}
\label{lemma:0'}
Let two (isomorphic, identically triangulated, but oppositely oriented) boundary components of manifold~$M'$ be glued into one surface~$\Gamma$ which lies in manifold~$M$ --- the result of gluing~$M'$. 
Then, the analogue of formula~(\ref{skleika}) gives zero:
\begin{equation}
\mathbf I_M = \frac{(-1)^{N_{\Gamma}} \tau_{\Gamma}^2}{\prod_{\textrm{over edges in }\Gamma} l^2} \int \mathbf I_{M'} \prod_{i\in \mathcal E_{\Gamma}} da_i^*\, da_i \equiv 0.
\label{skleika'}
\end{equation}
It is assumed in formula~(\ref{skleika'}) that the Grassmann variables at the edges of the two boundary components of~$M'$ are also identified when these components are glued into one surface~$\Gamma$.
\end{lemma}

\begin{proof}
It is not hard to see that function~$\Phi$ corresponding to our~$\mathbf I_M$ (recall that the connection between them is defined by formula~(\ref{I_M})) is generating for minors of~$f_3$ using which the torsion of the same complex~(\ref{complex-b}) as in the proof of lemma~\ref{lemma:petlya-0} is calculated (for rows of minor of~$f_2$, the edges in the union of minimal rigid constructions in $M'$ and~$\Gamma$ are taken). As, according to the mentioned proof, this complex is non-acyclic, all minors of~$f_3$, and thus $\mathbf I_M$ as well, vanish.
\end{proof}

\section{Geometric invariants and the modification of Atiyah's axioms for anti-commuting variables}
\label{sec:Atiyah-axioms}

We are going to compare the properties of our theory with axioms proposed by M.~Atiyah~\cite{atiyah,atiyah1} for a topological quantum field theory. As mentioned in the Introduction, the most interesting distinction from Atiyah's axioms understood literally arises when we glue a cylinder $\Sigma\times I$, where $\Sigma$ is a two-dimensional surface and $I=[0,1]$, into the closed manifold $\Sigma\times S^1$.

Atiyah's axioms, as applied to three-dimensional manifolds, require that to each \emph{two-dimensional} manifold~$\Sigma$ which can play the role of boundary for some three-dimensional manifold~$M$, should correspond a vector space~$V$, and any~$M$ with boundary~$\Sigma$ must determine a vector in this space; these must satisfy some very natural, from the viewpoint of mathematical physics, requirements, described below.

In our theory, as it is presented in this paper, $\Sigma$ is a closed \emph{triangulated} oriented manifold, equipped moreover with parameters --- ``Euclidean coordinates'' of boundary vertices. Additionally, we choose in every connected component of~$\Sigma$ a minimal rigid construction of edges, put in correspondence to each of the remaining edges~$i$ a pair of Grassmann variables $a_i,a_i^*$, and compose a ``generating function'' for the invariants of~$M$, depending on these variables. Certainly, we could well retain the minimal rigid construction and consider functions of Grassmann variables corresponding to all boundary edges, but this appears unnecessary: we have seen in section~\ref{sec:gluing} that our functions are enough for the main operation with which Atiyah's axioms deal --- the gluing of manifolds. There is no problem to consider the generating function an element of a (finite-dimensional) vector space of all functions of all $a_i,a_i^*$. As concerns the dependence on a triangulation, passing to a different triangulation corresponds, according to section~4 
of paper~\cite{I}, to a linear transform of generating functions. It can be shown that a change of minimal rigid construction leads to their linear transform as well. Thus, the choice of specific triangulation and minimal rigid construction can be regarded as a choice of basis in~$V$.

Next, according to Atiyah's axioms, if $\Sigma$ is a disjoint sum of $\Sigma_1$ and~$\Sigma_2$, then vector space~$V$ is the tensor product $V_1\otimes V_2$ of the corresponding spaces. 
As for our theory, we can say simply that, when we take such disjoint sum, free union of the sets of Grassmann generators is taken. If $\Sigma_1=\partial M_1$ and~$\Sigma_2=\partial M_2$, then the \emph{vector} corresponding to the free union of $M_1$ and~$M_2$ is equal, according to Atiyah, to the tensor product of vectors for $M_1$ and~$M_2$; in our theory, it also comes out with no problem as the product of two generating functions --- \emph{even} (this can be seen from formulas (\ref{Phi}) and~(\ref{I_M})) elements of the Grassmann algebra.

The further axioms, which Atiyah calls involutivity (passing to the dual space~$V^*$ when the orientation of~$\Sigma$ is changed) and multiplicativity 
(the behavior of invariants under a composition of cobordisms), find, together, a parallel in our formula~(\ref{skleika}). It gives exactly the description of the behavior of our invariants under a composition of cobordisms, giving exactly the ``convolution'' in variables related to the surface where the gluing goes (it is for such convolution that Atiyah needs involutivity).

For $\Sigma=\emptyset$ (closed~$M$) space~$V$ is one-dimensional; this agrees with one of Atiyah's axioms ensuring the nontriviality of the theory.\footnote{According to Martyushev's conjecture~\cite{M-diss}, our invariant for a closed~$M$ is expressed in terms of homology group $H_1(M)$.}

One more Atiyah's nontriviality axiom states that the identity operator in~$V$ must correspond to a cylinder $\Sigma\times I$, because the gluing of such cylinder to a manifold must not change its invariants. Hence, Atiyah derives that the invariant for $\Sigma\times S^1$ must be the trace of identical operator, that is, the dimensionality of~$V$. As for our theory, formula~(\ref{skleika}) by no means can be reduced to a product of linear operators, and (\ref{skleika'}) --- to taking a trace of a linear operator and, moreover, according to lemma~\ref{lemma:0'}, formula~(\ref{skleika'}) always gives zero. Can we still calculate the dimensionality of the vector space of generating functions for~$\Sigma$ in out theory?

The following lemma outlines an approach to solving this problem.

\begin{lemma}
Let $a=(a_1,\ldots,a_{2n})$ and~$b=(b_1,\ldots,b_{2n})$ be non-intersecting sets of equal even numbers~$2n$ of Grassmann generators. Consider a linear operator
\begin{equation}
A\colon\; f(b)\mapsto \int f(a) K(a,b)\, da_{2n}\ldots da_1,
\label{operator-ab}
\end{equation}
where kernel $K$ is an arbitrary function of $a$ and~$b$. Then its trace is
\begin{equation}
\Trace A = \int K(a,-a) \, da_{2n}\ldots da_1.
\label{sled-ab}
\end{equation}
\end{lemma}

\begin{proof}
It is enough to prove the lemma for $K(a,b)$ being monomials in $a$ and~$b$. This is done by a simple direct calculation; we illustrate it on the following example: if $K(a,b)=a_2a_3\ldots a_{2n}b_1$, then the only nonvanishing diagonal matrix element of $A$ appears for $f(b)=b_1$, and this element is~$1$ according to both (\ref{operator-ab}) and~(\ref{sled-ab}).
\end{proof}

So, to find the dimensionality of the vector space of generating functions, we must somehow change the signs at the variables belonging to the upper base of cylinder $\Sigma\times \{1\} \subset \Sigma \times I$. It turns out that this is quite a meaningful operation: one can produce an algebraic complex whose invariant, calculated according to formula~(\ref{inv-closed}), will give the desired. The detailed presentation of this, together with rather complicated explicit calculations for specific~$\Sigma$ of genus $1,2,\ldots$, deserves, in our opinion, a separate paper~\cite{M-prepar}. Here we point out that such complex will be \emph{twisted}~\cite{M-diss,M1}. To be exact, it uses a twisting of its differentials by the following representation of group $\pi_1(\Sigma\times S^1)$: if the projection of a loop onto~$S^1$ makes an odd number of revolutions, then the differentials \emph{change their signs}.

\section{Discussion}
\label{sec:discussion}

Recall that, algebraically, our invariants are based, on the one hand, on rather enigmatic differential formulas, whose structure imitates Pachner moves. These formulas are written in many of our works; the most striking seems to be, by now, the formula, proved in paper~\cite{kiev} using computer algebra.

On the other hand, we are guided by the algebraic theory of Reidemeister torsions. Note that other authors, too, pointed out at a connection between Reidemeister torsions and quantum field theory (although, as far as we know, not at a strict mathematical level): it is enough to remember the lectures by D.~Johnson~\cite{J} cited by Atiyah or a recent work by J.~Barrett and I.~Naish-Guzman~\cite{B-NG}.

As for the results of the present paper, it is remarkable from the physical point of view how anti-commuting variables, that is, essentially, fermions, appear ``by themselves'' from geometry. Putting it informally, one can see in this a cumulative action of two reasons. First, every boundary edge can either enter or not in each of our sets $\mathcal C$ and~$\mathcal D$ which are used in our invariant, similarly to a fermion that is either present or not in a given state. Second, our invariants are based on torsions, that is, a generalization of the determinant. The main role is played by minors of matrix~$f_3$, and under a gluing of manifolds, these matrices add up. What happens in this way to minors, cannot be reduced to simple operator product, because in trying to do so, problems with \emph{signs} of different summands would arise in the sum of products of minors. These sings, however, are taken into account automatically if we introduce generating functions of Grassmann variables and take Berezin integral in variables corresponding to the edges glued together.

We obtain a modification of Atiyah's axioms which can be called Atiyah--Berezin axioms. Note that Atiyah himself readily admits different modifications of his axioms.

For our plans of further work we see, on the one hand, such big tasks as building a discrete analogue of Chern--Simons invariants, and do this for manifolds of any dimension (recall once more the papers~\cite{33,24,15} devoted to four-dimensional manifolds, as well as the fact that algebraic complexes of our type can be built not only using Euclidean geometry~\cite{SL2',kiev,kkm}). On the other hand, a large amount of calculations for specific cases must be done, for instance, those mentioned in the end of section~\ref{sec:Atiyah-axioms}, and which can well lead to a deeper understanding of the theory.

\subsection*{Acknowledgements}

I would like to remember, with a few kind words, my scientific supervisor, the creator of superanalysis, Felix Alexandrovich Berezin, who pointed out to me the widely understood ``integrability'' as a field of scientific activity. This work has been partially supported by Russian Foundation for Basic Research, Grant no.~07-01-96005-r\_ural\_a.


\begin {thebibliography}{99}
\bibitem{J} D.~Johnson, A geometric form of Casson's invariant and its connection to Reidemeister torsion. Unpublished lecture notes.
\end{thebibliography}
\end{document}